\documentclass[11pt]{article}
\usepackage{geometry}             
\usepackage{graphicx}
\usepackage{amssymb}
\usepackage{epstopdf}
\usepackage{amsthm}
\usepackage{amsmath}
\usepackage{graphicx}
\usepackage{subcaption}
\usepackage{comment}
\usepackage{tikz,ifthen}
\usepackage{soul}
\usepackage{hyperref}

\newtheorem{definition}{Definition}
\newtheorem{theorem}{Theorem}
\newtheorem{corollary}{Corollary}
\newtheorem{lemma}{Lemma}

\usepackage{xcolor}

\title{On Grundy indices for complete geometric graphs}
\author{Dolores Lara\footnotemark[1], Christian Rubio-Montiel\footnotemark[2], Francisco Zaragoza\footnotemark[3]}
\date{\today}                                        

\begin{document}
\maketitle

\def\thefootnote{\fnsymbol{footnote}}
\footnotetext[1]{Departamento de Computaci{\' o}n, Centro de Investigaci{\' o}n y de Estudios Avanzados del Instituto Polit{\' e}cnico Nacional, Mexico, {\tt dolores.lara@cinvestav.mx}.}

\footnotetext[2]{Divisi{\' o}n de Matem{\' a}ticas e Ingenier{\' i}a, FES Acatl{\' a}n, Universidad Nacional Aut{\'o}noma de M{\' e}xico, Naucalpan, Mexico. {\tt christian.rubio@acatlan.unam.mx}.}

\footnotetext[3]{Departamento de Sistemas, Universidad Aut{\' o}noma Metropolitana Azcapotzalco, Mexico City, Mexico. {\tt franz@azc.uam.mx}.}

\begin{abstract}
    The pseudo-Grundy index of a graph is the largest number of colors that can be assigned to its edges, such that for every pair of colors $i,j$, if $i < j$ then every edge colored with color $j$ is adjacent to at least one edge colored with color $i$. This index has been widely studied. A geometric graph is a graph drawn in the plane such that its vertices are points in general position, and its edges are straight-line segments. In this paper, we extend the notion of pseudo-Grundy index for geometric graphs, and present results for complete geometric graphs.
\end{abstract}

\section{Introduction} 

Let $G = (V,E)$ be a graph. A drawing of $G$ in the plane where each of its vertices is a point, and every edge is drawn as a straight-line segment among these vertices, is a \emph{geometric graph} $\mathsf{G}$ of $G$. We assume that the set of points is in general position; that is, there are no three collinear points. A geometric graph whose vertices are in convex position is called a \emph{convex graph}. 

For geometric graphs, it is possible to define edge adjacencies in at least four different ways. We say that two edges \emph{intersect} if they have a common endpoint or if they cross. Two edges are \emph{disjoint} if they do not intersect. Depending on the type of adjacency that we consider, four different families of geometric graphs can be defined: those in which the edges intersect and their complement (disjoint); or those in which the edges cross and their complement (non-crossing). These incidence graphs were introduced in~\cite{MR2155418}.

We study colorings of geometric graphs. We say that the coloring of the edges of a graph $G$ is \emph{proper} if no two adjacent edges have the same color, and it is \emph{complete} if for every two distinct colors $i$ and $j$ used in the coloring, there exist adjacent edges of $G$ colored $i$ and $j$. Let $\mathsf{G}$ be a geometric graph defined from any of the four adjacency criteria described above. The \emph{chromatic index $\chi^{'}(\mathsf{G})$ of $\mathsf{G}$} is the smallest positive integer $k$ for which there exists a complete and proper $k$-coloring of the edges of $\mathsf{G}$. We say that a proper coloring of the edges of $\mathsf{G}$ is a \emph{Grundy coloring} if for every pair of colors $i,j$ such that $i < j$, every edge colored with color $j$ is adjacent to at least one edge colored with color $i$; every Grundy coloring is complete. The largest positive integer $k$ for which $\mathsf{G}$ has a Grundy $k$-coloring is the \emph{Grundy chromatic index of $\mathsf{G}$}, $\Gamma^{'}(\mathsf{G})$. If we drop the restriction that the Grundy coloring must be proper, then we get a related type of coloring known in the literature as \emph{pseudo-Grundy} coloring. The largest positive integer $k$ for which $\mathsf{G}$ has a pseudo-Grundy $k$-coloring is denoted by $\gamma^{'}(\mathsf{G})$ and is called the \emph{pseudo-Grundy chromatic index of $\mathsf{G}$}. 

We extend these definitions to graphs in the following way. Let $G$ be a graph. The \emph{geometric chromatic index} $\chi_g^{'}(G)$ of $G$ is the largest value $k$ for which there is a geometric graph $\mathsf{H}$ of $G$ such that $\chi^{'}(\mathsf{G}) = k$. Likewise, the \emph{geometric Grundy index} $\Gamma_g^{'}(G)$ and the \emph{geometric pseudo-Grundy index} $\gamma_g^{'}(G)$ of $G$ are defined as the smallest value $k$ for which a geometric graph $\mathsf{H}$ of $G$ exists, such that $\Gamma^{'}(\mathsf{G}) = k$ and $\gamma^{'}(\mathsf{G}) = k$, respectively.

The geometric chromatic index has been studied before. Let $\mathsf{K}^c_n$ denote the complete convex graph. In~\cite{MR2155418}, considering the adjacency criterion in which two edges are adjacent if they intersect, the authors proved that $\chi'(\mathsf{K}^c_n) = n$, and also that $n \leq \chi'_g(K_n) \leq cn^{3/2}$ for some constant $c > 0$. In the same paper, the adjacency criterion in which two edges are adjacent if they are disjoint was also considered; upper and lower bounds for $\chi_g'(K_n)$ were given. Later, for $\mathsf{K}^c_n$, the bounds were improved in~\cite{MR3862366} and \cite{MR3446120}; and finally in~\cite{jonsson2011exact,FJVW2018exact} the exact value of $\chi'(\mathsf{K}^c_n)$ was given. If $n \leq 16 $, then it is known that the exact value of $\chi'_g(K_n)$ is equal to $n-2$; this was recently proved in~\cite{MR4788271}. For the particular case in which the set of vertices is the so-called double-chain point configuration, the exact value of $\chi'(\mathsf{K_n})$ was given in~\cite{MR4104109}. In a different setting, in~\cite{MR4272849} the authors study the disjointness graph of a set of $n$ line segments in $\mathbb{R}^n$; they present an upper bound for the chromatic number of such a graph, among other results. In~\cite{MR4115538} the chromatic number of the disjointness graph of curves in $\mathbb{R}^n$ is studied.

Other geometric coloring parameters have also been studied. In~\cite{MR3461960} the authors study the geometric achromatic and pseudo-achromatic indices for complete geometric graphs, considering the criterion in which two edges are adjacent if they intersect. The achromatic index $\alpha'(\mathsf{G})$ of $\mathsf{G}$ is the largest integer $k$ for which there exists a complete and proper coloring of the edges of $\mathsf{G}$ using $k$ colors. The pseudoachromatic index $\psi'(\mathsf{G})$ of $\mathsf{G}$ is the largest integer $k$ for which there exists a complete coloring of the edges of $G$ using $k$ colors. The geometric achromatic index and the geometric pseudoachromatic index of a graph $G$, are defined as the smallest value $k$ for which a geometric graph $\mathsf{H}$ of $G$ exists such that $\alpha'(\mathsf{H})=k$ and $\psi'(\mathsf{H})=k$, respectively. They prove that $\alpha'(\mathsf{K}^c_n) = \psi'(\mathsf{K}^c_n)= \lfloor\frac{n^2+n}{4} \rfloor$, and also that
$0.0710 n^2 - \Theta(n) \leq \psi'_g(K_n) \leq 0.1781n^2 + \Theta(n)$. In~\cite{MR4341191}, the authors consider the criterion in which two edges are adjacent if they are disjoint, and present lower and upper bounds for the geometric achromatic and pseudo-achromatic indices.

In this paper, we study geometric Grundy and pseudo-Grundy indices for complete geometric graphs, under three of the four adjacency criteria. Our results are summarized in Table~\ref{tab:results}.

\begin{table}[]
    \centering
       \begin{tabular}{ c |c| c }
  Criterion & Convex position & General position\\
  \hline
  \hline
  \begin{tikzpicture}[baseline=(current bounding box.center)]
\draw[fill=black] (1,0) circle (1.5pt);
\draw[fill=black] (2,0) circle (1.5pt);
\draw[fill=black] (1,1) circle (1.5pt);
\draw[fill=black] (2,1) circle (1.5pt);
\node at (1.5,-0.5) {Crossing};
\draw[thick] (1,1) -- (2,0);
\draw[thick] (1,0) -- (2,1);
\end{tikzpicture}& 
  {
  $\begin{aligned}
 \frac{n^2}{12} &\leq \gamma'(\mathsf{K}^c_n)  \leq O \left( \frac{n^2}{\sqrt{24}} \right )
\end{aligned}$
}
& 
{$ \lfloor \frac{n^2}{400} \rfloor \leq \gamma_g'(K_n) \leq {\frac{n^2}{9}}$ }\\
  \hline
\begin{tikzpicture}[baseline=(current bounding box.center)]
\draw[fill=black] (0,0) circle (1.5pt);
\draw[fill=black] (1,0) circle (1.5pt);
\draw[fill=black] (0.5,1) circle (1.5pt);
\draw[fill=black] (2,0) circle (1.5pt);
\draw[fill=black] (3,0) circle (1.5pt);
\draw[fill=black] (2,1) circle (1.5pt);
\draw[fill=black] (3,1) circle (1.5pt);
\node at (1.5,-0.5) {Intersection};
\draw[thick] (0,0) -- (0.5,1);
\draw[thick] (1,0) -- (0.5,1);
\draw[thick] (2,1) -- (3,0);
\draw[thick] (2,0) -- (3,1);
\end{tikzpicture}& 
  {$\begin{aligned}
 \frac{n^2}8{}+\frac{n}{4} &\leq \Gamma'(\mathsf{K}^c_n) \\
      \gamma_g'(\mathsf{K}^c_n) &\leq O\left(\frac{n^2}{\sqrt{24}} \right)
\end{aligned}$}
& 
{$ \lfloor \frac{n^2}{400} \rfloor \leq \gamma_g'(K_n) \leq \frac{n^2}{9} + O(n)$} \\
  \hline
\begin{tikzpicture}[baseline=(current bounding box.center)]
\draw[fill=black] (0,0) circle (1.5pt);
\draw[fill=black] (0,1) circle (1.5pt);
\draw[fill=black] (1,0) circle (1.5pt);
\draw[fill=black] (1,1) circle (1.5pt);
\node at (0.5,-0.5) {Disjointness};
\draw[thick] (0,0) -- (0,1);
\draw[thick] (1,0) -- (1,1);
\end{tikzpicture} & 
& 
$\frac{n^2}{8}-O(n) \leq \gamma_g'(K_n) \leq \frac{n^2}{4} + O(n)$\\
\hline
\begin{tikzpicture}[baseline=(current bounding box.center)]
\draw[fill=black] (0,0) circle (1.5pt);
\draw[fill=black] (1,0) circle (1.5pt);
\draw[fill=black] (0.5,1) circle (1.5pt);
\draw[fill=black] (2,0) circle (1.5pt);
\draw[fill=black] (2,1) circle (1.5pt);
\draw[fill=black] (3,0) circle (1.5pt);
\draw[fill=black] (3,1) circle (1.5pt);
\node at (1.5,-0.5) {Non-crossing};
\draw[thick] (0,0) -- (.5,1);
\draw[thick] (.5,1) -- (1,0);
\draw[thick] (2,0) -- (2,1);
\draw[thick] (3,0) -- (3,1);
\end{tikzpicture}&  &  $\frac{n^2}{6} - O(n) \leq \gamma_g'(K_n) \leq \frac{n^2}{4} + O(n)$\\
  \hline
\end{tabular}
    \caption{Our results.}
    \label{tab:results}
\end{table}

We will proceed as follows. In the next section, we discuss the first two criteria: segment intersection and segment crossing. Section~\ref{sec:disjoint} is devoted to the study of the segment disjointness criterion. Finally, in Section~\ref{sec:non-crossing} we study the segment non-crossing criterion.

\section{Segment Intersection and Segment Crossing}\label{sec:intersecction-crossing}

In this section, we consider complete geometric graphs defined using two adjacency criteria: that in which two edges are adjacent if they intersect, that is, if the edges share a vertex or cross each other, and that in which two edges are adjacent if they cross each other. 

It is worth noticing that, under the first criterion, a clique of the graph is a straight line thrackle, and an independent set of edges is a plane matching. A straight line thrackle~\cite{MR277421} is a graph drawn in the plane so that its edges are straight line segments, and any two distinct edges meet at exactly one common vertex or cross. Under the segment intersection criterion, in any proper coloring of the graph, every thrackle with $k$ edges uses $k$ colors; the size of the largest independent edge set of the graph is an upper bound on the size of any proper chromatic class. Under the second criterion, the segment-crossing criterion, the cliques of the graph are crossing families, and an independent set of edges is a planar graph. A crossing family is a collection of line segments that have the property that every pair crosses each other~\cite{MR1289067}.

In what follows, we study the complete convex graph and give a lower bound for the geometric Grundy index and an upper bound for the geometric pseudo-Grundy index. Then, we consider the more general case in which the vertices of the complete geometric graph are not necessarily in convex position, and present lower and upper bounds for the pseudo-Grundy geometric index.
 
\subsection{Points in convex position} 

In this section, we consider the case in which the $n$ vertices of the complete geometric graph are in a convex position. We call this type of graph a \emph{complete convex  graph}, and we denote it by $\mathsf{K^{c}_{n}}$. Now, let $\mathsf{K^{c}_{n}}$ be a complete convex graph defined using any of the two intersection criteria, and let $\{1, \ldots, n\}$ be the vertices of the graph listed in clockwise order. In the rest of this subsection, we exclusively work with these types of graphs. It is important to note that all sums are taken modulo $n$, so for the sake of simplicity, we will not write this explicitly. The crossing pattern of the edge set of a complete convex graph depends only on the number of vertices, and not on their particular position. Therefore, without loss of generality, we assume that the point set of the graph corresponds to the vertices of a regular polygon.  

We prove a lower bound for the Grundy index of $\mathsf{K^{c}_{n}}$ and an upper bound for its pseudo-Grundy index, which, in turn, also sets an upper bound for its Grundy index.

\subsubsection{Lower bound}

First, we present the result for the intersection criterion; let $\mathsf{K^{c}_{n}}$ be a complete convex graph defined using this criterion. The following concept will be central in the proof of Theorem~\ref{thm:convex_lower}. For a given graph, we denote by $e_{i,j}$ the edge between the vertices $i$ and $j$. 

\begin{definition}
Given a set $J \subseteq \{1,2,  \ldots, \left \lfloor\frac{n}{2} \right \rfloor\}$, a \emph{circulant graph}  $C_n(J)$ is the graph with $n$ vertices labeled $\{1, 2, \ldots, n\}$ and a set of edges equal to $E(C_n(J)) = \{e_{i,j} : j-i \equiv k \pmod n, \text{or } j - i \equiv -k \pmod n, k \in J\}$. 
\end{definition}

The complete graph $K_n$ is the circulant graph $C_n(J)$ when $J = \{1, \ldots, \left \lfloor\frac{n}{2} \right \rfloor\}$.

In order to present a lower bound for the Grundy index of $\mathsf{K^{c}_{n}}$, we describe a (proper) Grundy coloring of its edge set. Let us describe the coloring before formally stating it. Consider the following partition of the edge set of $\mathsf{K_n}$: $E(\mathsf{K_n}) = C_n(\{1\}) \cup C_n(\{2\}) \cup \ldots \cup  C_n(\{\lfloor\frac{n}{2} \rfloor\}).$ Then, for every $i \in \{1, \ldots, \left \lfloor\frac{n}{2} \right \rfloor\}$, we describe a coloring of the edges of $C_n(\{i\})$ with $i+1$ colors, and prove that this coloring is a (proper) Grundy coloring. The coloring is as follows. First, we consider the first two edges of $C_n(\{1\})$, $e_{1,2}$ and $e_{2,3}$, assign color $1$ to $e_{1,2}$ and color $2$ to $e_{2,3}$. Now we repeat this pattern with the rest of the edges of $C_n(\{1\})$, using colors one and two alternately. We stop whenever adding one more edge to any of the chromatic classes makes the coloring no longer proper; see Figure~\ref{fig:circulantes}. We repeat this coloring method with each one of the other circulants in the partition, using $i+1$ new colors to color the first $i+1$ edges of the circulant $C_n(\{i\})$, and then coloring the rest of its edges using that same set of colors in a bottom-up greedy-coloring fashion. The greedy coloring algorithm is as follows. Given an ordering of the edges of $G$, consider each edge in this order and assign to the current edge the smallest possible color not already given to one of its neighbors. Every coloring produced by this algorithm is proper, complete, and Grundy \cite{MR2450569}.

\begin{definition}
Let $i$ be a positive integer such that $1 \leq i \leq \frac{n}{2}$, and let $e_{x,y}$ be an edge in a given circulant $C_n(\{i\})$. A \emph{rotation} of $e_{x,y}$ is an edge $e'_{x',y'}$ with $$x' = x + (i+1) c,$$ $$y' = y + (i+1) c,$$ with $c$ some positive integer such that $c \leq \frac{n-x-i}{i+1}$. 
\end{definition} 

\begin{lemma}\label{lema:rotation}
Let $e_{x,y}$ be any edge in a given circulant $C_n(\{i\})$ and let $e'_{x',y'}$ be any of its rotations. The edges $e_{x,y}$ and $e'_{x',y'}$ are disjoint. 
\begin{proof}
It follows from the definition of rotation that $x < x'$ and $y < y'$, and since $e$ and $e'$ $\in E(C_n({i}))$ then $x < y$, $x' < y'$, $y = x + i$, and $y < x'$. Therefore, $x < y < x' < y'$.
\end{proof}
\end{lemma}

\begin{theorem}
Let $\mathsf{K^{c}_{n}}$ be a complete convex graph of order $n$ defined from the edge-intersection criterion. The Grundy index $\Gamma'(\mathsf{K^{c}_{n}})$ of $\mathsf{K^{c}_{n}}$ has the following lower bound $$\frac{n^2}{8} + \frac{n}{4} \leq \Gamma'(\mathsf{K^{c}_{n}}).$$
\end{theorem}\label{thm:convex_lower}

\begin{proof}

Consider the following partition of the edge set of $\mathsf{K}^c_n$: $$\bigcup_{i=1}^{\left \lfloor\frac{n}{2} \right \rfloor} E(C_n(\{i\})).$$

For every $i \in \{1, \ldots, \left \lfloor\frac{n}{2} \right \rfloor\}$, we describe a coloring of the edges of $C_n(\{i\})$ with $i+1$ colors, and prove that this coloring is a (proper) Grundy coloring. 

Start with $i=1, k = 0$. Consider the set of the first $i+1$ edges of $C_n(\{i\})$, namely, $\{e_{1,(i+1)}, e_{2,(i+2)}, \ldots, e_{(i+1), i+(i+1)}\}$. Assign a different color to each of these edges, specifically, for $l \in \{1, \ldots, i+1\}$ assign the color $k+l$ to the edge $e_{l, (i+l)}$. Also, for $l \in \{1, \ldots, i+1\}$, assign the color $k+l$ to all possible rotations of $e_{l, (i+l)}$; see Figure~\ref{fig:circulantes-uno}. Repeat the above procedure for every $i \in \left \{1, \ldots, \left \lfloor\frac{n}{2} \right \rfloor \right \}$, increasing the value of $k$ by $i+1$ each time; see Figure~\ref{fig:circulantes-dos}.

We use $i+1$ colors for each circulant, with $i \in \{1, \ldots, \left \lfloor\frac{n}{2} \right \rfloor\}$, therefore, the total number of colors is $\sum_{i=1}^{\lfloor\frac{n}{2}\rfloor} i+1\geq \frac{n^2+2n}{8}$. Now we show that the coloring is a Grundy coloring. 

Two edges have the same color only if they belong to the same circulant and one is a rotation of the other. For Lemma~\ref{lema:rotation}, any edge and all its rotations are disjoint, and therefore the coloring is proper. To show that it meets the Grundy property, suppose to the contrary that there exists an edge $e' \in E(C_n({i'}))$ colored with color $c'$ such that it does not intersect any edge in the chromatic class $c$, with $c < c'$. Consider the set of edges in the chromatic class $c$, by construction all these edges belong to the same circulant graph, say $C_n(\{i\})$ and are the rotations of one starting edge, say $e$; see Figure~\ref{fig:samecolor}. Since $c < c'$, then $i < i'$. $e'$ does not intersect any edge with color $c$, therefore a subset of points of size $i'+1$ can be separated from the set of vertices without intersecting any edge with color $c$, but this is a contradiction since the size of the biggest of these subsets is $i-1$.

Note that not all edges of the graph have been colored; we color the remaining edges using the greedy-coloring algorithm. The coloring is still proper, and so is Grundy.
\end{proof}

\begin{figure}
    \centering
    \begin{subfigure}[b]{0.4\textwidth}
        \includegraphics[scale=.6]{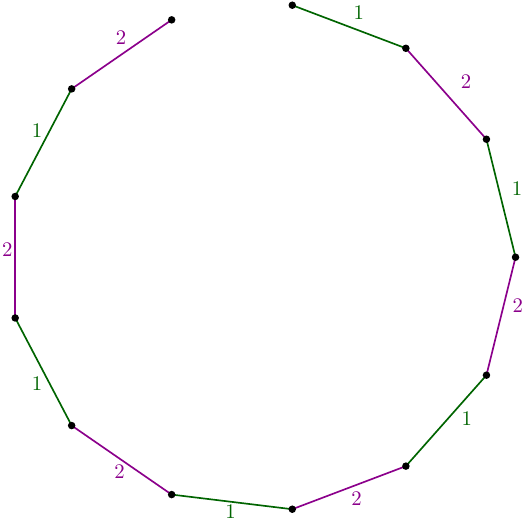}
        \caption{First circulant.} \label{fig:circulantes-uno}
    \end{subfigure}
    \begin{subfigure}[b]{0.4\textwidth}
        \includegraphics[scale=.6]{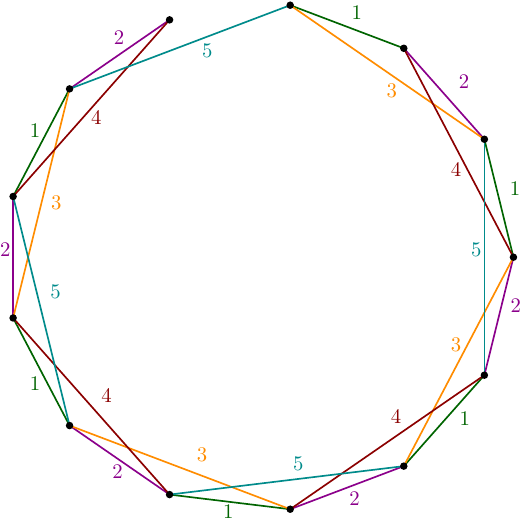}
        \caption{First and second circulant.} \label{fig:circulantes-dos}
    \end{subfigure}
\caption{The coloring of the first two circulants of the complete convex  graph with $13$ vertices.}
\label{fig:circulantes}
\end{figure}

\begin{figure}
\begin{center}
\includegraphics[scale=.7]{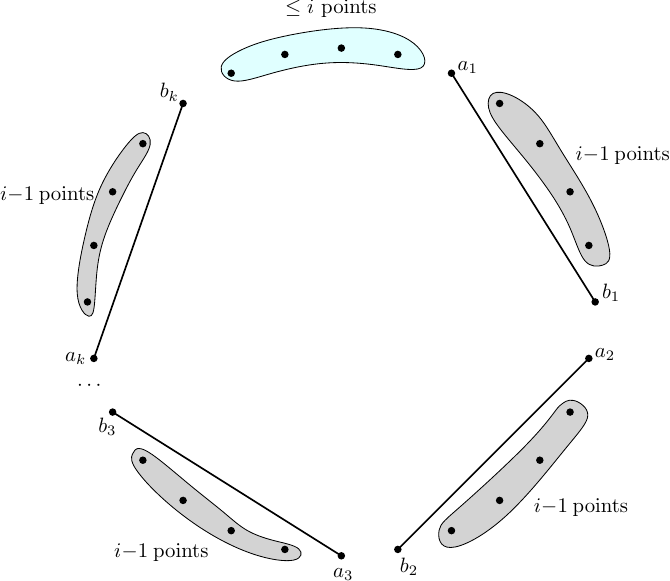}
\end{center}
\caption{A set of edges with the same color.} \label{fig:samecolor}
\end{figure}

Now we present a lower bound for the edge-crossing criterion.

\begin{theorem}
Let $\mathsf{K^{c}_{n}}$ be a complete convex graph of order $n$ defined from the edge-crossing criterion. The pseudo-Grundy index $\gamma'(\mathsf{K^{c}_{n}})$ of $\mathsf{K^{c}_{n}}$ has the following lower bound 
$$\frac{\left(n^2-16\right)}{12} \leq\gamma'(\mathsf{K^{c}_{n}}).$$
\end{theorem}

\begin{proof}
Let $\mathsf{K^{c}_{n}}$ be a complete convex graph of order $n$ defined from the edge-crossing criterion, and let $\{1, \ldots, n\}$ be the vertices of the graph listed in clockwise order. To avoid parity issues, we assume, without loss of generality, that $n$ is a power of two. Otherwise, consider the subset $\{1, 2, \ldots, m\}$, where $m$ is the biggest power of two less than $n$.

Consider the following partition of the set of vertices,
$V_{1,1} = \{1, \ldots, \frac{n}{4}\}$, 
$V_{1,2} = \{\frac{n}{4}+1, \ldots, \frac{2n}{4}\}$ ,
$V_{1,3} = \{\frac{2n}{4}+1, \ldots, \frac{3n}{4}\}$,
$V_{1,4} = \{\frac{3n}{4}+1, \ldots, n\}$. Color the edges between $V_{1,1}$ and $V_{1,3}$ using the $(\frac{n}{4})^2$ biggest colors; reuse these colors to color the edges between $V_{1,2}$ and $V_{1,4}$. We repeat this action several times, considering partitions with subsets of smaller size each time. That is, for the next step, we consider the eight subsets, each of size $\frac{n}{8}$, $V_{2,i} = \{\frac{(i-1)n}{8}+1, \ldots, \frac{in}{8}\}$, with $1 \leq i \leq 8$; Figure~\ref{fig:crossing} shows $V_{2,i}$ for odd $i$ in blue, and $V_{2,i}$ for even $i$ in red.  Color the edges between $V_{2,1}$ and $V_{2,3}$ with the next $(\frac{n}{8})^2$ biggest colors; reuse these colors to color the edges between $V_{2,i}$ and $V_{2,i+2}$, with $2 \leq i \leq 8$; sum is taken modulo $8$. For step $j$, we consider the $2^{j+1}$ subsets, each of size $\frac{n}{2^{j+1}}$, $V_{j,i} = \{\frac{(i-1)n}{2^{j+1}}+1, \ldots, \frac{in}{2^{j+1}}\}$, with $1 \leq i \leq 2^{j+1}$ and $3 \leq j \leq \log n -1$; all sums are taken modulo $2^{j+1}$. For each $j$, color the edges between $V_{j,1}$ and $V_{j,3}$ with the next 
$(\frac{n}{2^j})^2$ biggest colors, reuse these colors to color the edges between $V_{j,i}$ and $V_{j,i+2}$.

Next, we show that each pair of consecutive edges in the same chromatic class crosses each other; Figure~\ref{fig:crossing} shows one chromatic class of $\mathsf{K^{c}_{32}}$. Let $(i,j)$ be some edge that has been colored by the algorithm described above; its consecutive edge with the same color must be of the form $(i+s, j+s)$, where $s$ is the cardinality of the sets that define the two edges; that is $s \in \{n/4, n/8, \ldots, 1\}$. To prove our claim, it is sufficient to show that $i < i+s < j$, since the point set is in convex position. The first inequality is trivially true because $s > 0$; now, by definition, there must be at least $s$ points between $i$ and $j$, so the second inequality is also true. 

Now we show that the coloring is pseudo-Grundy. Let $\mathcal{C}$ be the set of edges colored with color $l$.  For each pair of consecutive edges in $\mathcal{C}$, consider their crossing point. These points, together with the set of edges $\mathcal{C}$, define a closed curve $\alpha$. Now, consider an edge $e$ colored with color $k$ such that $k > l$. $e$ crosses $\alpha$, because otherwise, by definition, $k$ cannot be greater than $l$.

The number of colors used is $\sum_{j=2}^{\log_2 n -1}(\frac{n}{2^j})^2 = \frac{1}{12} (n^2-16)$. We color the remaining edges using the greedy-coloring algorithm. This completes the proof.

\begin{figure}
    \centering
    \includegraphics[width=0.5\linewidth]{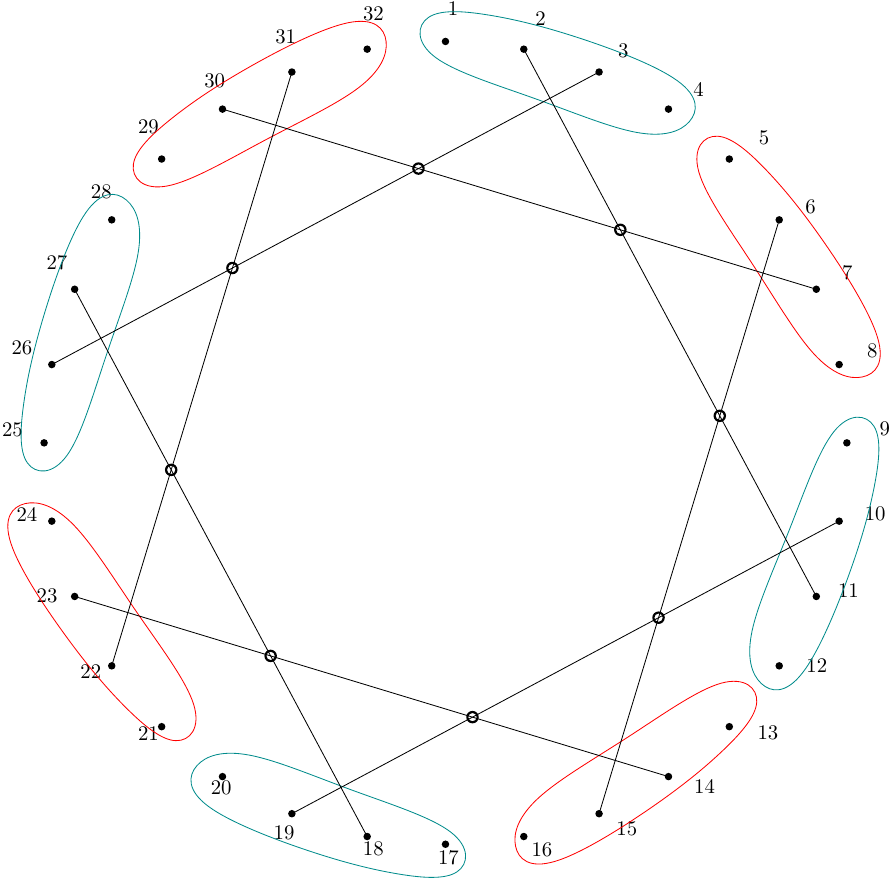}
    \caption{A chromatic class of $\mathsf{K^{c}_{32}}$}
    \label{fig:crossing}
\end{figure}

\end{proof}

\subsubsection{Upper bound}

The following result was given in~\cite{MR15796}.

\begin{theorem}
Any geometric graph with $n$ vertices and $n+1$ edges contains two disjoint edges.
\end{theorem}

Using this theorem the following result can be obtained, the order of a graph denotes the number of its vertices.

\begin{corollary}\label{cor}
Let $\mathsf{G}$ be a geometric graph of order $n$ defined from the edge-intersection or the edge-crossing criteria. There are at most $n$ chromatic classes of size one in any complete coloring of $\mathsf{G}$. In particular, there are at most $n$ chromatic classes of size one in any Grundy coloring of $\mathsf{G}$.
\end{corollary}

Consider two intersecting edges in $\mathsf{K}^c_n$ defined from either one of the two criteria, the intersection might occur at a common interior point (crossing) or at a common endpoint (at a vertex). Denote by $m$ the total number of intersections that occur at the vertices, since $\mathsf{K}^c_n$ is the complete graph then $ m = n \binom{n-1}{2}$, if we are using the edge-intersection criterion or $m=0$ if we are using the edge-crossing criterion. Denote by $cr$ the total number of intersections that occur at interior points, since $\mathsf{K}^c_n$ is the complete graph and the set of points is in convex position, then $ cr = \binom{n}{4} $, in both cases. 

Using these observations, we prove the following result:

\begin{theorem}\label{thm:upper_bound}
Let $\mathsf{K}^c_n$ be a complete convex graph of order $n$ defined from the edge-intersection or the edge-crossing criteria. The pseudo-Grundy number $\gamma(\mathsf{K}^c_n)$ of $\mathsf{K}^c_n$ is at most $O(\frac{n^2}{\sqrt{24}})$.

\begin{proof}
Let $\gamma$ be the pseudo-Grundy number of $\mathsf{K}^c_n$, and let $\{1, 2, \ldots, \gamma\}$ be the set of colors used in a Grundy coloring of its edges, using $\gamma$ colors.

Corollary~\ref{cor} implies that the maximum number of chromatic classes of size one in any Grundy coloring of the graph is $n$. Suppose that all other chromatic classes are of size at least two. We show the result by counting the minimum number of incidences required in a Grundy coloring with these properties. It is not hard to see that a coloring that minimizes the required number of incidences must be one in which the chromatic classes of size one are precisely $\{\gamma, \gamma-1, \ldots, \gamma-(n-1)\}$. Let us count the number of incidences used by each edge. Let $i \in \{0, 1, \ldots, n-1\}$ and $j \in \{n, n+1, \ldots \gamma-1\}$. An edge with color $\gamma - i$ requires at least $\gamma -i -1$ incidences with smaller colors, while an edge with color $\gamma - j$ requires at least $\gamma - j - 1$ incidences with smaller colors. Therefore, the number of incidences required is at least $\binom{\gamma}{2} + \binom{\gamma-n}{2} \leq  cr \leq m + cr$. After doing some algebra, we obtain the desired result.
\end{proof}
\end{theorem}

\subsection{Points in general position} 
In this subsection, we consider point sets in general position in the plane and present lower and upper bounds for the geometric pseudo-Grundy index of complete geometric graphs. We consider the segment-intersection and the segment-crossing adjacency criteria. Recall that the geometric pseudo-Grundy index of a graph $G$ is defined as: 
$${\gamma'_g} (G) = min \{\gamma_1 (\mathsf{G}) : \mathsf{G} \text{ is a geometric graph of G}\}.$$

\subsubsection{Lower bound}

To achieve the lower bound, we use the following result. Consider $k$ sets of points in general position in the plane, $\{A_1, \ldots, A_k\}$. We say that a set of points $T$, of the form $T =\{p_1, \ldots, p_k\}$, is a \emph{transversal} of the $k$ sets $A_i$, if $p_1 \in A_1, \ldots, p_k \in A_k$. $T$ is in convex position if each $p_i$ is a vertex of the convex hull of $T$. The authors of~\cite{MR1755231} proved that every set of $n$ points in general position in the plane contains four subsets $A_1, \ldots, A_4$, each one with cardinality at least $\lfloor \frac{n}{20} \rfloor$, and such that every transversal of the sets $A_i$ is convex. To prove their result, the authors give a partition of the plane using a collection of lines that results in the convex sets of the desired cardinality. For a similar theorem, see~\cite{MR1608874}. 

\begin{theorem}\label{thm:traversal}(Nielsen and Sabo)
Every sufficiently large finite set $S$ of $n$ points in general position in the plane contains $4$ subsets $A_1$, $A_2$, $A_3$, $A_4$, each of cardinality at least $\lfloor \frac{n}{20}\rfloor$, such that every set $\{a_1, a_2, a_3, a_4\}$ with $a_1 \in A_1, a_2 \in A_2, a_3 \in A_3$, and $a_4 \in A_4$ is in convex position.
\end{theorem}

Using the above theorem, we can establish the following result; we present a coloring of the edges of any $\mathsf{K_n}$.

\begin{theorem}
Let $K_n$ be a complete graph with $n$ vertices. Under both the segment crossing and the segment intersection criteria, the pseudo-Grundy geometric index of $K_n$ has the following lower bound.
$$\left \lfloor \frac{n^2}{400} \right \rfloor \leq \gamma'_g({K_n}).$$

\begin{proof}
To prove the result we construct a set of $\left \lfloor \frac{n^2}{400} \right \rfloor$ pairwise crossing edges.

Let $\mathsf{K_n}$ be a complete geometric graph of $n$ points in general position in the plane. Take the four subsets $A_1, A_2, A_3, A_4$ of $V(\mathsf{K_n})$ guaranteed by Theorem~\ref{thm:traversal}. Using these four subsets, we can always construct a set of edges $M$ and a set of edges $M'$ with the following characteristics: $M, M' \subseteq \bigcup_{1\leq i \leq 4}A_i$, $|M| = |M'| = \binom{\frac{n}{20}}{2}$, and such that every edge in $M$ intersects every edge in $M'$. 

Color the edges of $M$ with colors from $1$ to $\binom{\frac{n}{20}}{2}$ and repeat for $M'$. This coloring is Grundy, and we have assigned at least $\lfloor \frac{n^2}{400}\rfloor$ colors. Color the remaining edges using the greedy coloring algorithm. The result follows. 
\end{proof}
\end{theorem}

\subsubsection{Upper bound} 

We obtain an upper bound for $\gamma'_g(K_n)$, for the segment intersection and the segment crossing criteria, using the following definitions and remarks.

Let $\mathsf{G}$ be a geometric graph and let $e$ be an edge of $\mathsf{G}$. We define the \emph{geometric edge-degree} of $e$ as the number of edges in $\mathsf{G}$ that intersect (share a common vertex or cross) $e$. We call the largest degree among the edges of $\mathsf{G}$ the \emph{maximum geometric edge-degree} of $\mathsf{G}$ and denote it by $\Delta_1(\mathsf{G})$. We define the \emph{crossing index} $\overline{cr}_1(G)$ of a graph $G$, as the smallest value $k$ for which there exists a geometric drawing $\mathsf{G}$ of $G$ such that $\Delta_1 (\mathsf{G}) = k$.

Let $\mathsf{K}_n$ be a complete geometric graph, defined using the segment intersection or the segment crossing criteria, and let $k$ be its Grundy number. Consider a Grundy coloring of the edges of $\mathsf{K}_n$, using $k$ colors, and let $e$ be an edge to which the color $k$ has been assigned. Since the coloring is a Grundy coloring, for every $i \in \{1, 2, \ldots, k-1\}$, there must be an edge with color $i$ intersecting $e$.
Thus $\Delta_1 (\mathsf{K}_n) \geq \gamma'(\mathsf{K}_n)-1$ and then
$$\Delta_1(\mathsf{K}_n)\geq\overline{cr}_1(K_n)\geq \gamma'_g(K_n)-1.$$

Two of the authors of this paper proposed the problem to Bernardo Ábrego and Silvia Fernández-Merchant, they observed that the parameter $\Delta_1(\mathsf{K}_n)$ is called in the literature the rectilinear local crossing number of $K_n$ and is denoted by $\overline{lcr}(K_n)$. They proved the following~\cite{MR3663492}.

\begin{theorem}(Ábrego and Fernández)
If $n$ is a positive integer, then 

\begin{equation*}
    \overline{lcr}(K_n) =
    \begin{cases}
        \frac{1}{9}(n-3)^2 & \text{ if } n \equiv 0 \pmod 3\\
        \frac{1}{9}(n-1)(n-4) & \text{ if } n \equiv 1 \pmod 3\\
        \frac{1}{9}(n-2)^2- \lfloor \frac{n-2}{6} \rfloor & \text{ if } n \equiv 2 \pmod 3, n \neq \{8,14\}\\
    \end{cases}
\end{equation*}

In addition, $\overline{lcr}(K_8) = 4$ and $\overline{lcr}(K_{14}) = 15$.
\end{theorem}

This result directly implies our next theorem.
\begin{theorem} 

Let $K_n$ be a complete graph with $n$ vertices. Under both the segment crossing and the segment intersection criteria, the pseudo-Grundy geometric index of $K_n$ has the following upper bound.
$$\gamma'_g(K_n)\leq \frac{n^2}{9}+\Theta(n).$$
\end{theorem}


\section{Segment Disjointness} \label{sec:disjoint}

In this section, we consider complete geometric graphs defined using the adjacency criterion in which two edges are adjacent if they are disjoint. We say that two edges are disjoint if they do not intersect. Under this criterion, the cliques of the graph are plane matchings and its independent edge sets are thrackles. In what follows, we present upper and lower bounds for the geometric pseudo-Grundy index of points in general position.

Let $\mathsf{G}$ be a geometric graph of order $n$ defined by the disjointness criterion. There are at most $\frac{n}{2}$ chromatic classes of size one in any complete coloring of $\mathsf{G}$. In particular, there are at most $\frac{n}{2}$ chromatic classes of size one in any Grundy coloring of $\mathsf{G}$. From this fact, we can get the following result for the geometric pseudo-Grundy index of the complete graph. We omit the proof because it is similar to that of Theorem~\ref{thm:upper_bound}.

\begin{theorem}
Let $n\geq 4$ be a natural number, the geometric pseudo-Grundy index of the complete graph can be bounded from above as follows
\[\gamma_g'(K_n)\leq \frac{n^2}{4}+\theta(n).\]
\end{theorem}

Before presenting the lower bound, we recall the definition of halving-line. If $P$ is a set of $n$ points in general position in the plane, when $n$ is even a \emph{halving-line} of $P$ is a line through two points of $P$ leaving $(n-2)/2$ points of $P$ on each side; when $n$ is odd, a halving-line leaves $(n-3)/2$ and $(n-1)/2$ points on each side. Any set of points in general position in the plane has at least one halving-line.

\begin{theorem}
Let $n\geq 4$ be an integer, then the geometric pseudo-Grundy index of the complete graph can be bounded from below as follows
\[\gamma'_g(K_n)\geq \frac{n^2}{8}-\theta(n).\]
\begin{proof}
Let $\mathsf{K}_n$ be a complete geometric graph, and let $S$ be its vertex set. Consider a halving-line $\ell$ of $S$, without loss of generality, suppose that there are equal or more points to the right of $\ell$ than to the left; let this number of points be $k$. Color the edges to the right of $\ell$ using colors $\{1, \ldots, {k \choose 2}\}$, and color the edges to the left using the same set of colors (minus $k-1$ if $n$ is odd); one color for each edge. Consider some edge colored with color $j$, call this edge $e$, and suppose that both of its vertices are to the left of $\ell$; for every color $i$, between $1$ and $j-1$, there is an edge colored $i$ whose vertices are to the right of $\ell$ and therefore that is disjoint to $e$. This partial coloring is pseudo-Grundy; we are using at least $(n-3)(n-1)/8$ colors. The rest of the edges are colored using the greedy-coloring algorithm. 
\end{proof}
\end{theorem}

\section{Segment Non-Crossing}\label{sec:non-crossing}
In this last section, we consider complete geometric graphs defined by using the adjacency criterion in which two edges are adjacent if they do not cross. Under this criterion, the cliques of the graph are plane graphs and its independent edge sets are crossing families. 


\begin{theorem}
Let $K_n$ be the complete graph with $n\geq 3$ vertices. If we consider the non-crossing criterion, then the pseudo-Grundy geometric index of $K_n$ is bounded from above by
\[\gamma'_g(K_n)\leq \frac{n^2}{4}+\theta(n).\]
\end{theorem}
\begin{proof}
Let $\mathsf{K}_n$ be a complete geometric graph. Any complete coloring of $\mathsf{K}_n$ has at most $3n-6$ classes of size one (the number of edges in a maximal plane graph), then \[\gamma'_g(K_n)\leq \frac{\tbinom{n}{2}-(3n-6)}{2}+3n-6= \frac{n(n+5)}{4}-3,\]
and the result follows.
\end{proof}

For the Grundy index of complete graphs, we prove $\frac{n^2}{6}+\Theta(n)\leq \Gamma'_g(K_n)$, using
the following lemma and another result from the literature. 

\begin{lemma}\label{lem:triangle-edge}
Let $S$ be a set of points in general position in the plane. Let $T$ be any triangle with vertices in $S$, and $e$ be any segment that connects two points in $S$. Choose $T$ and $e$ so that they do not have vertices in common. $T$ contains at least one edge disjoint to $e$.
\begin{proof}
    The line expanded by the edge $e$ divides the plane into two half-planes; one half-plane contains exactly one vertex of $T$. The degree of this vertex is exactly two; therefore, there are at most two edges of $T$ that cross $e$.
\end{proof}
\end{lemma}

We draw the reader's attention to the following result on graph decompositions; see \cite{MR2071903,MR382030}. Table~\ref{tab:partition} illustrates the different cases in the theorem.

\begin{theorem}(H. Hanani)\label{thm:partition}
    Let $K_n$ be the complete graph with $n$ vertices. There exists a set $F \subset E(K_n)$ such that $E(K_n) \setminus F$ can be partitioned into triangles. More precisely, if $n\equiv 1, 3 \pmod 6$, then $F = \emptyset$, if $n\equiv 0, 2 \pmod 6$, then $F$ induces a perfect matching in $K_n$, if $n\equiv 4 \pmod 6$, then $F$ induces a spanning forest of $n/2 + 1$ edges in $K_n$ with all vertices having an odd degree (a tripole), and if $n\equiv 5 \pmod 6$, then $F$ induces a $4$-cycle. 
\end{theorem}

\begin{table}
    \centering
\begin{tabular}{ c | c | c | c}
  $n \equiv 1 \text{ or } 3 \pmod 6$ & $0 \text{ or } 2 \pmod 6$ & $5 \pmod 6$ & $4 \pmod 6$\\
  \hline
  \hline
Empty set
& 
\begin{tikzpicture}[baseline=(current bounding box.center)]
\node at (0.5,-0.5) {Perfect matching};
\def \n {3}
\foreach \s in {0,...,\n}
{
\draw[fill=black] (0,\s) circle (1.5pt);
\draw[fill=black] (1,\s) circle (1.5pt);
\draw[thick] (0,\s) -- (1,\s);
}
\draw[fill=black] (0.5,1.5) circle (1.0pt);
\draw[fill=black] (0.5,1.7) circle (1.0pt);
\draw[fill=black] (0.5,1.3) circle (1.0pt);
\end{tikzpicture}
& 
\begin{tikzpicture}[baseline=(current bounding box.center)]
\node at (0.5,-0.5) {$4$-cycle};
\draw[fill=black] (0,0) circle (1.5pt);
\draw[fill=black] (0,1) circle (1.5pt);
\draw[fill=black] (1,0) circle (1.5pt);
\draw[fill=black] (1,1) circle (1.5pt);
\draw[thick] (0,0) -- (0,1);
\draw[thick] (0,1) -- (1,1);
\draw[thick] (1,1) -- (1,0);
\draw[thick] (1,0) -- (0,0);
\end{tikzpicture}
& 
\begin{tikzpicture}[baseline=(current bounding box.center)]
\node at (0.5,-0.5) {Tripole};
\def \n {3}
\draw[fill=black] (0.5,5) circle (1.5pt);
\draw[fill=black] (0.5,4) circle (1.5pt);
\draw[fill=black] (0,4) circle (1.5pt);
\draw[fill=black] (1,4) circle (1.5pt);
\draw[thick] (0.5,5) -- (0.5,4);
\draw[thick] (0.5,5) -- (0,4);
\draw[thick] (0.5,5) -- (1,4);
\foreach \s in {0,...,\n}
{
\draw[fill=black] (0,\s) circle (1.5pt);
\draw[fill=black] (1,\s) circle (1.5pt);
\draw[thick] (0,\s) -- (1,\s);
}
\draw[fill=black] (0.5,1.5) circle (1.0pt);
\draw[fill=black] (0.5,1.7) circle (1.0pt);
\draw[fill=black] (0.5,1.3) circle (1.0pt);
\end{tikzpicture}
\\
\hline
\end{tabular}
    \caption{The four different cases of Theorem~\ref{thm:partition}.}
    \label{tab:partition}
\end{table}

Now we are ready to present our lower bound.

\begin{theorem}\label{teo:remainder} 
Let $K_n$ be the complete graph, under the non-crossing criteria the Grundy geometric index of $K_n$ has the following lower bound,
\[\frac{n^2}{6}+\Theta(n)\leq \Gamma'_g(K_n)\]
\end{theorem}
\begin{proof}

Let $S$ be a set of $n$ points in general position in the plane. Consider the set $F$ obtained by applying Theorem~\ref{thm:partition} to a complete graph $K_n$. Draw $F$ in $S$ so that there are no crossings between its elements. It is not hard to prove that such a drawing always exists, just order the points in $S$ from left to right, and let $S=\{s_1, \ldots, s_n\}$ be the ordered set. If $F$ is a perfect matching draw the edges $(s_1,s_2), (s_2,s_3), \ldots, (s_{n-1}, s_n)$; if $F$ is a tripole, choose the first three points and draw the $K_{1,3}$ graph, and then draw the perfect matching $(s_4,s_5), (s_6,s_7), \ldots, (s_{n-1}, s_n)$; finally, if $F$ is a $4$-cycle, just draw the edges $(s_1,s_2), (s_2,s_3), (s_3,s_4), (s_4,s_1)$.

Without loss of generality, let $\mathsf{K}_n$ be the complete graph drawn in the plane as described above. Theorem~\ref{thm:partition} guarantees that the set of edges of $\mathsf{K}_n \setminus F$ can be partitioned into $k$ triangles. Color each of these triangles with a different color, using colors $\{1, 2, \ldots, k\}$. 
This partial coloring is pseudo-Grundy, since Lemma~\ref{lem:triangle-edge} implies that for every two triangles in the partition there are always two parallel edges, one from each triangle; however, the coloring is not proper. Now we must color $F$.

Color each edge in $F$ with a different color, using colors $\{k+1, \ldots, k+|F|\}$. By construction, every pair of edges in $F$ is parallel or shares exactly one vertex; therefore, $F$ induces a clique. This implies that the coloring of $F$ is Grundy. Now consider an edge $e$ in $F$ colored with color $j$, and let $T$ be a triangle in the partition, colored with color $i$, $i < j$. There are two possible situations; either $e$ and $T$ do not share any vertex or $e$ is incident to one vertex of $T$. In the first case, Lemma~\ref{lem:triangle-edge} guarantees that there is at least one edge in $T$ parallel to $e$, and its chromatic classes are adjacent. In the latter case, the condition is met by definition. Therefore, the coloring is pseudo-Grundy.

The number of colors used is 

\begin{equation*}
 k + |F|  = 
  \begin{cases}
    \frac{(n-1)n}{6} & \text{if } n\equiv 1, 3 \pmod 6\\
    \frac{(n+1)n}{6} & \text{if } n\equiv 0, 2 \pmod 6\\
     \frac{(n-1)n+16}{6}  & \text{if } n\equiv 4 \pmod 6\\
    \frac{(n+1)n+4}{6} & \text{if } n\equiv 5 \pmod 6
  \end{cases}
\end{equation*}
\end{proof}

\begin{figure}
    \centering
    \includegraphics[width=0.6\linewidth]{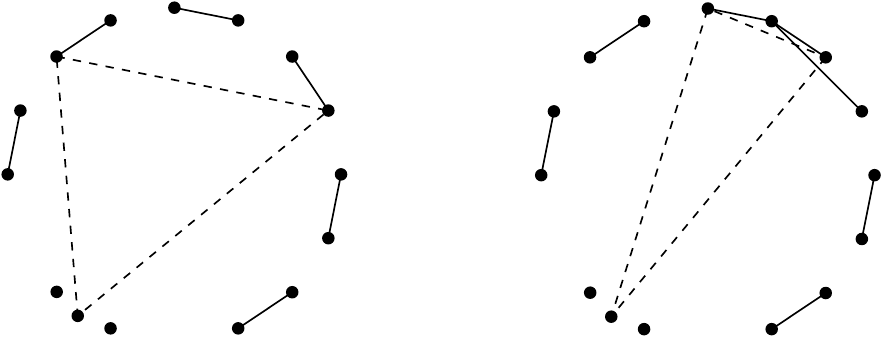}
    \caption{Configurations of the color classes of size one arising from $F$ and a dashed triangle of $K \setminus F$ in the proof of Theorem~\ref{teo:remainder}.}
    \label{fig:remainder}
\end{figure}



\section*{Acknowledgments}

Part of the work was done during the Taller de Geometr{\' i}a, Combinatoria y Algoritmos, held at Campus-Azcapotzalco, Universidad Aut{\' o}noma de Metropolitana, CDMX, Mexico on April 14-16, 2014. Part of the results of this paper was announced at XVI Spanish Meeting on Computational Geometry in Barcelona, Spain on July 1-3, 2015, see \cite{lara2015grundy}.


\bibliographystyle{unsrt}
\bibliography{biblio}

\end{document}